\documentclass[12pt,english,leqno]{amsart}
\usepackage[T1]{fontenc}
\usepackage[latin9]{inputenc}
\usepackage{amsthm}
\usepackage{amssymb}

\makeatletter
\numberwithin{equation}{section}
\numberwithin{figure}{section}
\theoremstyle{plain}
\newtheorem{thm}{\protect\theoremname}
  \theoremstyle{plain}
  \newtheorem{prop}[thm]{\protect\propositionname}
  \theoremstyle{plain}
  \newtheorem{cor}[thm]{\protect\corollaryname}

\makeatother

\usepackage{babel}
  \providecommand{\corollaryname}{Corollary}
  \providecommand{\propositionname}{Proposition}
\providecommand{\theoremname}{Theorem}

\begin{document}

\title{remarks on metrizability of dual groups}

\author{Fr\'ed\'eric Mynard and Mikhail Tkachenko}

\maketitle
\global\long\def\F{\mathcal{F}}
\global\long\def\G{\mathcal{G}}
\global\long\def\H{\mathcal{H}}
\global\long\def\U{\mathcal{U}}
\global\long\def\S{\mathcal{S}}
\global\long\def\filter{\mathbb{F}}
\global\long\def\lm{\mathrm{lim}}
\global\long\def\adh{\mathrm{adh}}
\global\long\def\N{\mathcal{N}}
\global\long\def\A{\mathcal{A}}
\global\long\def\B{\mathcal{B}}
\global\long\def\then{\Longrightarrow}
\global\long\def\cl{\mathrm{cl}}
\global\long\def\Seq{\mathrm{Seq}}

\date{$\today$}
\begin{abstract}
We examine sufficient conditions for the dual of a topological group
to be metrizable and locally compact, improving on results of \cite{nonmetrdualangelic}.
\end{abstract}

\section{Introduction and preliminaries}

If $G$ is an abelian topological group, its dual $\hat{G}$ is the
set of continuous group homomorphisms into the one-dimensional torus
$\mathbb{T}$, endowed with the compact-open topology. 

Two families $\A$ and $\B$ of subsets of a set $X$ are said to
\emph{mesh, }in symbols $\A\#\B$, if $A\cap B\neq\varnothing$ whenever
$A\in\A$ and $B\in\B$. Thus, if the set $\mathbb{F}X$ of filters
on a set $X$ is ordered by inclusion, two filters $\F$ and $\G$
admit a supremum $\F\vee\G$ if and only if $\F\#\G$. We do not distinguish
between a sequence $(x_{n})_{n\in\omega}$ and the filter generated
by its tails. In particular, if $\H$ is a family of subsets of $X$
and $(x_{n})_{n\in\omega}$ is a sequence on $X$, the notation $\H\#(x_{n})_{n\in\omega}$
means that $H\cap\{x_{n}:n\geq k\}\neq\varnothing$ for every $H\in\H$
and $k\in\omega$. 

Recall that a topological space $X$ is\emph{ sequential }if every
sequentially closed subset is closed. The space is \emph{Fr\'echet-Urysohn,
}or simply\emph{ Fr\'echet,} if whenever $x\in X$ and $A\subset X$
with $x\in\cl A$, there is a sequence on $A$ that converges to $x$.
It is \emph{strongly Fr\'echet} if whenever $(A_{n})_{n\in\omega}$
is a decreasing sequence of subsets of $X$ with $x\in\bigcap_{n\in\omega}\cl A_{n}$
there is $x_{n}\in A_{n}$ for each $n$ with $x_{n}\to_{n}x$, equivalently
if
\[
\forall\H\in\mathbb{F}_{1},\,\adh\H\subseteq\adh_{\Seq}\H,
\]
where $\mathbb{F}_{1}$ denotes the class of countably based filters,
$\adh\H:=\bigcap_{H\in\H}\cl H$ is the adherence of the filter $\H$,
and $\adh_{\Seq}\H:=\bigcup_{(x_{n})_{n}\#\H}\lim(x_{n})_{n}$ is
its sequential adherence. We will also denote by $\cl_{\Seq}$ the
sequential closure, that is, the closure in the coarsest sequential
topology that is finer than the original. 

We are interested in metrization results for dual groups akin to:
\begin{thm}
\label{th:angelic} \cite{nonmetrdualangelic}Let $G$ be a metrizable
topological group. If $\hat{G}$ is Fr\'echet, then it is metrizable
and locally compact. 
\end{thm}
To this end, let us review its simple proof: A product of a strongly
Fr\'echet space and a metrizable space is (strongly) Fr\'echet (e.g.,
\cite{michael}), and $\hat{G}$ is strongly Fr\'echet as a Fr\'echet
topological group \cite{nyikos}, so that $G\times\hat{G}$ is Fr\'echet,
hence sequential. It is in particular a $k$-space (%
\footnote{A topological space $X$ is a $k$-space if a subset $C$ is closed
whenever $C\cap K$ is closed for every compact subset $K$ of $X$.%
}), so that, in view of \cite[Proposition 1.2]{dunford}, $\hat{G}$
is locally compact. But a locally compact topological group with countable
tightness (in particular a sequential locally compact one) is metrizable
\cite{AVAM}. Hence $\hat{G}$ is locally compact and metrizable. 

Thus, the proof hinges on $G\times\hat{G}$ being a $k$-space, and
$\hat{G}$ being of countable tightness. In this note, which can be
considered as an appendix to \cite{nonmetrdualangelic}, we will consider
sufficient conditions on $G$ and $\hat{G}$ to achieve this result
that are different and often weaker than those in Theorem \ref{th:angelic}.

\section{Strongly sequential groups }

The proof outlined above relies on the following:
\begin{thm}
\cite{michael}\label{th:sFrechet} Let $X$ be a topological space.
The following are equivalent: 
\begin{enumerate}
\item $X$ is strongly Fr\'echet; 
\item $\lm\F=\bigcap_{\H\in\mathbb{F}_{1}}\adh_{\mathrm{Seq}}\H$ for every
filter $\F$ on $X$; 
\item $X\times Y$ is strongly Fr\'echet for every bisequential space $Y$; 
\item $X\times Y$ is Fr\'echet for every metrizable atomic topology. 
\end{enumerate}
\end{thm}

\emph{Strongly sequential} spaces are in some sense to sequential
spaces like strongly Fr\'echet spaces are to Fr\'echet spaces: 
\begin{thm}
\cite{sseq}\label{th:sseq} Let $X$ be a $T_{1}$ topological space.
The following are equivalent: 
\begin{enumerate}
\item $X$ is strongly sequential; 
\item $\lm\F=\bigcap_{\H\in\mathbb{F}_{1}}\cl_{\mathrm{Seq}}(\adh_{\mathrm{Seq}}\H)$
for every filter $\F$ on $X$; 
\item $X\times Y$ is strongly sequential for every bisequential space $Y$; 
\item $X\times Y$ is sequential for every metrizable atomic topology. 
\end{enumerate}
\end{thm}
The proof outlined in the Introduction thus applies virtually unchanged
to the effect that:
\begin{prop}
\label{prop:variantsseq} Let $G$ be a metrizable topological group.
If $\hat{G}$ is strongly sequential, then it is metrizable and locally
compact. 
\end{prop}
Since Fr\'echet topological groups are strongly Fr\'echet and strongly
Fr\'echet spaces are strongly sequential, Proposition \ref{prop:variantsseq}
formally generalizes Theorem \ref{th:angelic}. 

Note that while Theorems \ref{th:sFrechet} and \ref{th:sseq} highlight
similarities between strongly Fr\'echet and strongly sequential spaces,
Proposition \ref{prop:variantsseq} points to an important difference:
While Fr\'echet-Urysohn topological groups are automatically strongly
Fr\'echet, a sequential topological group does not need to be strongly
sequential. Indeed, there are metrizable groups $G$ such that $\hat{G}$
is sequential but not metrizable \cite{nonmetrdualangelic}. In view
of Proposition \ref{prop:variantsseq}, such dual groups are sequential
but not strongly sequential.

We will now see that Proposition \ref{prop:variantsseq} generalizes
Theorem \ref{th:angelic} only formally, as both are in fact instances
of a general metrization theorem for topological groups that is not
specific to dual groups. Indeed, if $G$ is metrizable, then $\hat{G}$
is a hemicompact $k$-space \cite{chascometr}, \cite{aussmetr}.
Thus, Theorem \ref{th:angelic} follows from:
\begin{thm}
\cite{hemicompactgroups}\label{th:hemicompactgroups} A hemicompact
Fr\'echet-Urysohn topological group is locally compact and Polish. 
\end{thm}
The local compactness (hence also metrization) part of this result
follows in turn from a more general fact. 

Recall from \cite{sseq2} that a topological space is called a \emph{Tanaka
space }if 
\[
\tag{Tanaka condition}\H\in\mathbb{F}_{1},\,\adh\H\neq\varnothing\then\adh_{\mathrm{Seq}}\H\neq\varnothing.
\]
Of course, every strongly sequential space, in particular every strongly
Fr\'echet space, is Tanaka. Thus every Fr\'echet topological group is
Tanaka. Therefore Theorem \ref{th:hemicompactgroups} is an instance
of Corollary \ref{cor:hemicompactgroups} below, and Proposition \ref{prop:variantsseq}
follows from Corollary \ref{cor:hemicompactgroups} as well, showing
that Theorem \ref{th:angelic} and Proposition \ref{prop:variantsseq}
are essentially the same.
\begin{thm}
\label{thm:hemiTanaka} A Hausdorff hemicompact Tanaka space has a
point with a compact neighborhood.\end{thm}
\begin{proof}
Let $(K_{n})_{n\in\omega}$ be an \emph{increasing }sequence of compact
subsets such that each compact subset of $X$ is contained in some
$K_{n}$. If there is $n$ with $\mathrm{int}K_{n}\neq\varnothing$,
we are done. Otherwise, each $K_{n}$ has empty interior, so that
its complement is dense. Thus $(X\setminus K_{n})_{n\in\omega}$ is
a decreasing sequence of sets with $\bigcap_{n\in\omega}\cl(X\setminus K_{n})=X$,
but there is no convergent sequence meshing with\emph{ }$(X\setminus K_{n})_{n\in\omega}$.
Indeed, if there was such a sequence, there would be a subsequence
$(x_{n})_{n\in\omega}$ with $x_{n}\in X\setminus K_{n}$ for each
$n$ that converges to a point $l$. But $K:=\{x_{n}:n\in\omega\}\cup\{l\}$
would then be a compact subset of $X$ that is not contained in any
of the $K_{n}$'s. Thus $X$ is not Tanaka.\end{proof}
\begin{cor}
\label{cor:homogTanka} A homogeneous Hausdorff hemicompact Tanaka
space is locally compact.
\end{cor}

\begin{cor}
\label{cor:hemicompactgroups}A Hausdorff hemicompact topological
group of countable tightness is locally compact and metrizable whenever
it is Tanaka.\end{cor}
\begin{proof}
This follows from Corollary \ref{cor:homogTanka} and the fact that
a locally compact topological group of countable tightness is metrizable
\cite{AVAM}.
\end{proof}

\section{Productively Fr\'echet groups}

\textit{Productively Fr\'echet} spaces are those whose product with
every strongly Fr\'echet space is (strongly) Fr\'echet \cite{JM.prodFrechet}.
To describe them, we first need to define strongly Fr\'echet filters.
A filter $\F$ on $X$ is \emph{strongly Fr\'echet }if 
\[
\H\in\mathbb{F}_{1},\H\#\F\then\exists\G\in\mathbb{F}_{1}:\G\geq\H\vee\F.
\]
Let $\mathbb{F}_{1}^{\triangle}$ denote the class of strongly Fr\'echet
filters. A topological space is \emph{productively Fr\'echet }if 
\[
\forall\H\in\filter_{1}^{\triangle},\,\adh\H\subseteq\adh_{\Seq}\H.
\]

\begin{thm}
\label{thm:prodF} \cite{JM.prodFrechet} The following are equivalent:
\begin{enumerate}
\item $X$ is productively Fr\'echet;
\item $X\times Y$ is strongly Fr\'echet for every strongly Fr\'echet space
$Y$;
\item $X\times Y$ is Fr\'echet for every strongly Fr\'echet space $Y$.
\end{enumerate}
\end{thm}
Thus it is clear that the proof of Theorem \ref{th:angelic} applies
with virtually no change to the effect that:
\begin{thm}
\label{th:myvariant2} 
\begin{enumerate}
\item If $G$ is a productively Fr\'echet topological group and $\hat{G}$
is Fr\'echet-Urysohn, then $\hat{G}$ is metrizable and locally compact.
\item If $G$ is a Fr\'echet topological group and $\hat{G}$ is productively
Fr\'echet, then $\hat{G}$ is metrizable and locally compact.
\end{enumerate}
\end{thm}
It is observed in \cite{JM.prodFrechet} that the $\Sigma$-product
of $\mathfrak{c}$ many copies of $\mathbb{R}$ is a non-metrizable
productively Fr\'echet topological group.

Note that since there are productively Fr\'echet group whose dual is
not hemicompact, Theorem \ref{th:myvariant2} no longer follows from
Theorem \ref{th:hemicompactgroups}. To see that such groups exist,
first note that the $\Sigma$-product of $\mathfrak{c}$ many copies
of $\mathbb{R}$ is a locally convex Fr\'echet-Urysohn topological vector
space, in particular a locally quasi-convex (by, e.g., \cite[Proposition 6.5]{aussmetr})
$k$-group. It is therefore \emph{subreflexive}, that is, 
\begin{eqnarray*}
i:G & \to & \hat{\hat{G}}\\
g & \mapsto & \left\langle g,\cdot\right\rangle 
\end{eqnarray*}
is an embedding.

Moreover, we have: 
\begin{prop}
\label{prop:subreflex} \cite[Proposition 5.10]{characterdual} A
subreflexive topological group $G$ is metrizable if and only if $\hat{G}$
is hemicompact. 
\end{prop}
As already noted, the $\Sigma$-product of $\mathfrak{c}$ many copies
of $\mathbb{R}$ is a productively Fr\'echet non-metrizable group. In
view of Proposition \ref{prop:subreflex}, its dual is not hemicompact. 

In conclusion, while Theorem \ref{th:angelic} and Proposition \ref{prop:variantsseq}
are not specific to dual groups (but are true for any hemicompact
group), the variants given in Theorem \ref{th:myvariant2} seem to
be specific to dual groups, even though the proof is essentially identical.

\address{F. Mynard, Mathematics, New Jersey City University, Kartnoutsos building
- Room 506 2039 John F. Kennedy Boulevard Jersey City, New Jersey
07305, USA}

\address{M. Tkachenko, Universidad Autnoma Metropolitana - Iztapalapa Av.
San Rafael Atlixco 186, Col. Vicentina C.P. 09340 Delegacin Izapalapa,
Mxico, D.F.}
\end{document}